\documentclass{amsart}
\usepackage{url}

\usepackage{graphicx}
\usepackage{datetime}
\usepackage{color}
\usepackage{amsmath,amsfonts,amsthm,amssymb}
\usepackage{epstopdf}
\usepackage[all,cmtip]{xy}
\usepackage{accents}
\newtheorem{theorem}{Theorem}
\newtheorem{proposition}[theorem]{Proposition}
\newtheorem{definition}[theorem]{Definition}
\newtheorem{corollary}[theorem]{Corollary}
\newtheorem{lemma}[theorem]{Lemma}
\newtheorem{remark}[theorem]{Remark}
\newtheorem*{thm}{Theorem \ref{thm_main}} 
\newtheorem*{thm3}{Theorem \ref{thm_classify}} 
\newtheorem*{thm2}{Theorem \ref{thm_main2}}
\numberwithin{equation}{section} \numberwithin{theorem}{section}

\newcommand{\R}{\mathbb{R}}
\newcommand{\N}{\mathbb{N}}
\newcommand{\Z}{\mathbb{Z}}
\newcommand{\circo}{\accentset{\circ}}
\newcommand{\e}{\varepsilon}

\newcommand{\mbb}{\mathbb}
\newcommand{\C}{\mbb{C}}

\newcommand{\mc}{\mathcal}
\newcommand{\la}{\langle}
\newcommand{\ra}{\rangle}

\def\XXint#1#2#3{{\setbox0=\hbox{$#1{#2#3}{\int}$}
     \vcenter{\hbox{$#2#3$}}\kern-.5\wd0}}

\newcommand{\ov}{\overline}

\DeclareMathOperator{\Res}{Res}
\DeclareMathOperator{\Per}{Per}

\DeclareMathOperator{\dist}{dist}

\DeclareMathOperator{\Area}{Area}

\begin{document}
\title[Branched Willmore spheres]{Branched Willmore spheres}
\author{Tobias Lamm}
\address[T.~Lamm]{Institut f\"ur Mathematik, Goethe-Universit\"at Frankfurt, Robert-Mayer-Str. 10, 60054 Frankfurt, Germany}
\email{lamm@math.uni-frankfurt.de}
\author{Huy The Nguyen}
\address[H.~Nguyen]{Mathematics Institute,
Zeeman Building,
University of Warwick,
Coventry, CV4 7AL, UK}
\email{H.T.Nguyen@warwick.ac.uk}
\thanks{The second author is supported by The Leverhulme Trust.}


\begin{abstract}
In this paper we classify branched Willmore spheres with at most three branch points (including multiplicity), showing that they may be obtained from complete minimal surfaces in $\R ^ 3$ with ends of multiplicity at most three. This extends the classification result of Bryant. We then show that this may be applied to the analysis of singularities of the Willmore flow of non-Willmore spheres with Willmore energy $ \mc {W} ( f ) \leq 16\pi$.
\end{abstract}

\maketitle
\section{Introduction}
Let $ \Sigma $ be a closed immersed surface $ f : \Sigma \rightarrow \R ^ 3 $, then the Willmore functional is defined as 
\begin{align}\label{willmore_functional} 
 \mc { W } (f) =  \int_\Sigma |H|^2 d \mu _g 
\end{align}
where $H=\frac12 (\kappa_1+\kappa_2)$ is the (scalar) mean curvature and $\mu$ is the induced area measure. The variational Euler-Lagrange operator is defined as 
 \begin{align}\label{eqn_EL}
W ( f) = \triangle H +  2H( H^2 - K).
\end{align}
A smooth immersion $f$ is called a Willmore surface if it is a critical point of the Willmore functional and therefore a solution of the equation $ W (f) = 0 $. In the paper \cite{Bryant1984}, Bryant studied smooth Willmore spheres in $\mbb{S}^3 $ and showed that after composing with a stereographic projection such surfaces are M\"obius transforms of complete minimal surfaces with planar ends in $\R^3$. In particular, he showed that Willmore spheres have quantised energy with $ \mc {W} ( f) = 4 \pi k, k \in \mbb N$ where $ k =1$ corresponds to the round sphere and the values $ k = 2, 3 $ are not allowed, since there are no minimal surfaces with planar ends corresponding to these energy bounds. In this paper we will extend Bryant's theorem to classify Willmore spheres in $ \R^ 3 $ with singular points. Note that singular points of a Willmore surface, even if they are graphical, are not removable. The standard example is one half of a catenoid composed with an inversion. Near the origin, the surface is $ C ^ {1, \alpha} $ for all $ \alpha <1$ but it is never $ C^{1,1}$ (see e.g. \cite{Kuwert2004}). Hence Bryant's classification does not extend to cover this case. However, we will be able to prove a classification of branched Willmore spheres. We give a precise definition of a branched Willmore immersion in Definition \ref{branchWillmore} but essentially a branched Willmore immersion is a branched immersion $f:\Sigma \to \R^3$ which smoothly solves the Willmore equation away from an at most finite set of branch points $\{p_1,\ldots,p_l\}$ with finite density $m_i$, $1\le i \le l$.

We will collect together these points as the divisor of the surface and denote it by $D = \sum_{i=1}^{l}m_i p_i$. Furthermore by carefully analysing branch points of Willmore surfaces, using the results of \cite{Kuwert2004} and \cite{kuwert07}, we will prove the following theorem, 
\begin{theorem}\label{thm_main}
Let $f: \mbb{S}^2 \rightarrow \R ^ 3 $ be a branched Willmore immersion with $|D| \leq 3$, then $f$ is either umbilic or the M\"obius transform of a branched complete minimal surface. In particular, this implies that the Willmore energy is quantised as $\mc {W} ( f) = 4 \pi k$.  
\end{theorem}
In principle, we can classify all such minimal surfaces by using the Weierstrass-Enneper representation. However, we will content ourselves with the following classification result
  
\begin{theorem}\label{thm_classify}
Let $f: \mbb{S}^2 \rightarrow \R ^ 3 $ be a branched Willmore immersion with $|D| \leq 3$ and $ \mc {W}(f) < 16 \pi$. Then one of the following cases occurs ($p_i\in \mbb{S}^2$, $1\le i \le 3$):
\begin{enumerate}
\item $D = \emptyset$ and $f$ is an isometric immersion of the round sphere with $ \mc{W} (f) = 4 \pi$.

\item $D =p_1+ p_2$ and $ f$ is the M\"obius transform of an embedding of a catenoid with $\mc{W} ( f) = 8\pi$.

\item $D = 3p_1$ and $f$ is the M\"obius transform of an immersion of Enneper's minimal surface with $\mc{W}(f) = 12\pi.$

\item $ D =p_1+ m_2p_2+ m_3p _3 , m _i \in\{0,1\} $ and $ f$ is the M\"obius transform of a trinoid with $ \mc{W}(f) = 12 \pi$.     

\end{enumerate}
\end{theorem}
Note that this classification contrasts with Bryant's theorem in \cite{Bryant1984}. In the case of smooth Willmore immersions of the sphere in $\R^3$ the values $ \mc W (f) = 8 \pi, 12 \pi$ are not allowed.  

A trinoid is a complete minimal surface of finite total curvature $\int_{\Sigma} |A|^ 2 = 16 \pi$ which is conformal to $ \mbb{S}^ 2 \backslash \{ p_1, p_2, p_3 \}$. These surfaces were completely described in \cite{L'opez1992}. The hypothesis $ |D| \leq 3$ implies that the surface admits either up to three multiplicity one singularities, one multiplicity two singularity and a multiplicity one singularity or one multiplicity three singularity. The above theorem tells us that neither the second possibility nor one multiplicity one singularity can occur. 
\begin{remark}\label{energy}
If the singularities of the surface occur with the same image point then the hypothesis of Theorem \ref{thm_classify} is equivalent to the energy bound $\mc {W}( f) <16 \pi $. This occurs for branched Willmore surfaces that appear as singularity models for the Willmore flow but in general is not true. For example inversions of minimal herissons, that is minimal surfaces with branch points with $\int K = -4\pi$, are Willmore surfaces with branch points with different image points, see \cite{Rosenberg1988}.
\end{remark}
Apart from the intrinsic interest of a classification as in Theorem \ref{thm_main}, \ref{thm_classify}, Willmore spheres with singularities naturally appear as energy bubbles in the energy identity for sequences of Willmore surfaces, see e.g. \cite{Bernard2011}. 

As an application of the above theorem, we will classify singularities of the Willmore flow of spheres subject to an initial energy bound.

The Willmore flow is given by 
\begin{align*}
 \frac { \partial f } {\partial t } = - W (f), \quad f ( x, 0) = f _0( x), 
\end{align*}
where $f_0:\Sigma \to \R^3$ is some given initial immersion.
In this paper, we will consider the Willmore flow of non-Willmore spheres with Willmore energy  $\mc {W}(f_0) \leq 16\pi$. The following theorem is a convergence result for the Willmore flow.
\begin{theorem}[{\cite[Theorem 5.2]{Kuwert2004}}]
Let $f_0 : \mbb{S}^2 \rightarrow \R ^ 3$ be a smooth immersion with Willmore energy 
\begin{align*}
\mc{W}(f_0) \leq 8\pi.
\end{align*}
Then the Willmore flow with initial data $f_0$ exists smoothly for all times and converges to round sphere.
\end{theorem}
\begin{remark}
The conclusion remains true for a smooth immersion $f_0 : \mbb{S}^2 \rightarrow \R^4$ with $\mc{W}(f_0) < 8\pi$ by combining the results of Kuwert and Sch\"atzle \cite{Kuwert2004} with the point removability result of Rivi\`ere \cite{Rivi`ere2008}.
\end{remark}
In the case of initial data with larger Willmore energy we do not expect smooth convergence to a round sphere since singularities do form \cite{Mayer2002}, \cite{Blatt2009}. It was shown by Kuwert and Sch\"atzle \cite{Kuwert2001}, \cite{Kuwert2004} (see also \cite{Kuwert2011}) that singularities at the first singular time for the Willmore flow are modelled by properly immersed, complete Willmore surfaces with finite total curvature. More precisely they showed that if there is a singularity of the flow then one can obtain by blowup, blowdown or translation a sequence of solutions to the Willmore flow that converges locally smoothly to a properly immersed, complete Willmore surface. After obtaining a complete Willmore surface, one can apply an inversion and then such a surface is a closed Willmore surface with possible branch points which are all mapped to one single point. The key to analysing singularities of the Willmore flow then is the classification result for branched Willmore spheres described in Theorem \ref{thm_classify}. The theorem for singularities of the Willmore flow that we will prove is 
\begin{theorem}\label{thm_main2}
Let $f_0: \mbb{S}^2 \rightarrow \R ^ 3 $ be a smooth immersion of a non-Willmore sphere with Willmore energy
\begin{align*} 
\mc{W} ( f_0) \leq 16 \pi.
\end{align*} 
If the maximal Willmore flow $f:\mbb{S}^2\times [0,T)\to \R^3$ with initial value $f_0$ does not converge to a round sphere then there exist sequences $ r _j, t _ j\nearrow T $ where
$ r_j \rightarrow \infty, r_j \rightarrow 0  $ or $ r_j \to 1 $ and a rescaled flow 
  \begin{align*}
 f _ j : \Sigma \times \left[ -\frac{ t_j}{ r_j^4} , \frac{ T - t_j} { r _j ^ 4 } \right) \rightarrow \R^3, f_j( p, \tau) = \frac{ 1}{ r_j} ( f ( p, t_j+ r_j^4 \tau) - x_j) 
\end{align*}
such that $f_j$ converges locally smoothly to either a catenoid, Enneper's minimal surface or a trinoid. In particular, if $\mc{W}(f_0) \leq 12\pi$ then either the maximal Willmore flow converges to a round sphere or $f_j$ converges locally smoothly to a catenoid.  
\end{theorem}
\begin{remark}
A similar result for rotationally symmetric initial immersions satisfying $\mc W (f_0)\leq 12\pi$ has been previously obtained by Blatt \cite{Blatt2009}.  
\end{remark}

The paper is set out as follows. In section \ref{sec_preliminaries}, we gather together facts and notations required for the remainder of the paper.  In section \ref{sec_willmore_spheres}, we classify branched Willmore surfaces with $ |D| \leq 3$. They key to this theorem is to show that a certain meromorphic four form which has poles on the divisor $D$ has at worst poles of second order. Once this is proven, the classification theorem follows from suitable modifications and improvements of results of Bryant \cite{Bryant1984} and Eschenburg \cite{Eschenburg}. 
In section \ref{sec_analysis} we consider the Willmore flow.  We use the singularity analysis of Kuwert and Sch\"atzle \cite{Kuwert2011} in order to show that if the flow does not converge to a round sphere, then in general we can obtain by blowup a Willmore surface with finitely many singular points. We then consider the Willmore flow of spheres with Willmore energy $ \mc { W} ( f_0) \leq 16 \pi$ where $ f_0$ is not a Willmore sphere. From the assumption we conclude that the Willmore energy must drop immediately. We will then prove that our hypotheses show that the singularity of the Willmore flow is modelled by a branched Willmore sphere with $|D|\leq 3$ which allows us to apply the classification theorem proved above.  

\section{Preliminaries}\label{sec_preliminaries}
In the following, we will require various formulae for geometric quantities under M\"obius transforms, in particular inversions. Such formulae hold for $W^{2,2}$ conformal immersions, which are defined as follows. 
\begin{definition}[{\cite{Kuwert2010}}]
Let $ \Sigma$ be a Riemann surface. A map $ f \in W^{2,2}_{loc}( \Sigma, \R^n)$ is called a conformal immersion if in some local coordinates $ (U, z )$, the metric $ g_{i j} = \langle \partial _i f , \partial_j f \rangle $ is given by
\begin{align*}
g_{ij} = e^{2u} \delta _{ij}, \quad u \in L^\infty_{loc} (U).
\end{align*} 
The set of all $ W^{2,2}$-conformal immersions of $\Sigma$ is denoted $W^{2,2}_{conf} ( \Sigma, \R^n)$.
\end{definition}
 For $W^{2,2}$ conformal immersions we have the weak Liouville equation, 
\begin{align*}
\int_{U} \langle Du, D\varphi \rangle = \int_{U} K_g e ^ {2u} \varphi \quad \forall \varphi \in C^\infty_0(U).
\end{align*}
This is shown to hold in \cite{Kuwert2010}. 

\begin{definition}[Branched conformal immersion]  \label{defn_branched}
 A map $ f: \Sigma \rightarrow \R^{n}$ is called a branched conformal immersion  (with locally square integrable second fundamental form) if $ f \in W ^ {2,2}_ {conf} ( \Sigma \backslash \mc{S}, \R^n)$ for some discrete $ \mc S \subset \Sigma$ and if for each $ p \in \mc {S}$ there exists a neighbourhood $ \Omega_p $ such that in local conformal co-ordinates
 \begin{align*}
\int_{\Omega_p\backslash \{p\} } |A|^2 d\mu_g < \infty.   
\end{align*}
Moreover, we either require that $ \mu_{g} (\Omega_p \backslash \{p\} )< \infty $ or that $ p $ is a complete end. 
\end{definition}  
It was proved in Kuwert-Li \cite{Kuwert2010} and M\"uller-\v{S}ver\'ak \cite{Muller1995} that about each point in the set $\mc{S}$ we can construct local conformal coordinates. In terms of these coordinates the assumption that $ \mu_g ( \Omega_p \backslash \{ p\})<\infty $ or that $ p$ is a complete end is equivalent to assuming that either $\int_{\Omega_p \backslash \{ p\} }e ^{\pm2u} d x < \infty$ and we also note that the set of branched conformal immersions is closed under inversions. In particular, in the paper \cite{Nguyen2011}, it is shown that a suitable inversion sends complete ends to finite area branch points. In the following we therefore only consider finite area branch points. 

Now we consider a branched conformal immersion $f:\Sigma \to \R^3$ and since we are only interested in the local behaviour of the immersion we assume without loss of generality that $\mc{S}=\{p\}$. Furthermore we assume that $f:\Sigma \backslash \{p\}\to B^3_1(0) \backslash \{0\}$ with $0\in \text{spt} \mu$, where
\[
 \mu:=f(\mu_g)=\Big(x\mapsto \mathcal{H}^0(f^{-1}(x))\Big) \mathcal{H}^2 \lfloor f(\Sigma)
\]
is the weight measure of the integral $2$-varifold associated with $f$ (see e.g. \cite{Kuwert2004}).

Next we choose a punctured neighbourhood of $p$ which is conformally equivalent to $D_1\backslash \{0\}$ where $D_1=\{ z\in \R^2 |\,\ |z|<1\}$. 
Altogether this shows that $f$ is a $W^{2,2}$-conformal immersion from $D_1\backslash \{0\}$ to $B^3_1(0)\backslash \{0\}$ with $\int_{D_1\backslash \{0\} } |A|^2 d\mu+\mu_g(D_1\backslash \{0\})<\infty$ and hence a result of Kuwert-Li \cite{Kuwert2010} (see also Lemma A.4 in \cite{Rivi`ere2010}) shows the existence of a number $m\in \N$ such that $\theta^2 (\mu,0) =m$. Moreover $f$ is smooth away from the origin.

Finally we note that Kuwert-Sch\"atzle \cite{kuwert07} showed that in the present situation the conformal parametrization $f$ and the conformal factor $u$ satisfy 
\begin{align*}
 |f(z)|\simeq |z|^{m} \ \ \ \text{and}\ \ \ |\nabla f(z)|\simeq e^{u(z)} \simeq |z|^{m-1},
\end{align*}
for all $z\in D_1 \backslash \{0\}$. Moreover, by a result of Kuwert-Li \cite{Kuwert2010} we have for all $z\in D_1\backslash \{0\}$
\begin{align}
 u(z)=(m-1)\log |z| +v(z)+h(z), \label{expansionu}
\end{align}
where $h:D_1\to \R$ is a smooth harmonic function and $v\in C^0\cap W^{1,2} \cap W^{2,1} (D_1)$ is a solution of the equation
\begin{align}
 -\Delta v=K_g e^{2u}. \label{pseudoliouville}
\end{align}

Altogether we can apply the following result of Kuwert-Sch\"atzle (which is a combination of Theorem 3.5 in \cite{Kuwert2001} with Lemma 4.1 in \cite{Kuwert2004}, and Theorem 1.1 in \cite{kuwert07}). 
\begin{theorem}\label{thm_expansion}
Let $ \Sigma$ be an open surface and $ f : \Sigma \rightarrow B^3_\delta(0)\backslash \{0\} \subset \R ^ 3 $ be a smooth Willmore immersion that satisfies for some $m\in \N$
\begin{align*}
&0 \in \text{spt} \mu, \\
&\theta^2 (\mu,0) =m\ \ \ \text{and}\\ 
&\int_{ \Sigma} |A|^2 d\mu < \infty.
\end{align*}

Then the tangent cone of $ \mu$ at $0\in \R^3$ is unique and is a finite union of planes, and $\mu$ is given by multivalued graphs over these planes in a neighbourhood of the origin. Moreover, there exists a conformal parametrization of $\Sigma$ which we again denote by $f:D_1\backslash \{0\} \to \R^3 \backslash \{0\}$, such that for all $z\in D_1 \backslash \{0\}$ and all $\e>0$ there exists $C_{\e}>0$ so that the second fundamental form and its derivatives satisfy an estimate of the form
\begin{align*}
 |\nabla^k A( z) |  \leq C _{\e} | z| ^{1-k-m-\e} \quad \forall \e > 0. 
\end{align*}
Moreover, in the case $m=1$, $\mu$ is a $W^{2,p}\cap C^{1,\alpha}$-embedded unit density surface at $0$ for all $1\le p<\infty$ and all $0<\alpha<1$ and for the scalar mean curvature we have the expansion
\begin{align}
H(z)=H_0 \log|z|+W^{2,p}_{\text{loc}}, \label{expansionH}                                                                                                                                                                                                               
\end{align}
where $H_0\in \R$.

Furthermore, locally around the origin, the immersion has the following representation,
\begin{align}\label{eqn_loc_expansion}
f(z) &= \Re( a z^{m}) + O ( |z|^ {m + \gamma} )      \\
\partial _ 1 f ( z) - \imath \partial_ 2 f ( z ) &=  m  a z ^{ m -1} + O ( | z| ^ {m -1+ \gamma} )   
\end{align}
where $ \gamma > 0$ and $ a \in \mbb{C} \backslash \{0\}$. 
\end{theorem} 

We remark that in higher codimensions the corresponding result has recently been obtained by Bernard-Rivi\`ere \cite{Bernard2011a}.

We now give a precise definition of a branched Willmore surface. 
\begin{definition} [Branched Willmore immersions and the divisor]\label{branchWillmore} 
Let $\Sigma$ be a closed Riemann surface. We say that an immersion $f:\Sigma \to \R^3 $ is a branched Willmore immersion if $\int_\Sigma |A|^2 d\mu <\infty$ and if there exists an at most finite set of points $ \mc {S}=\{p_1,\ldots, p_l\}$ such that $f$ is a smooth Willmore immersion on $\Sigma \backslash \mc{S}$. As above we assume that $\mc{S}$ only consists of finite area branch points. 

Each branch point $p_k\in \mc{S}$ has a well-defined multiplicity $m_k\in \N$. We denote by $ D = \sum_{k=1}^l m _k p_k$ the divisor of the branched Willmore surface and we let $ |D| =\sum_{k =1}^ l  m_k$ be the order of the divisor.              
\end{definition} 
\begin{remark}
By results of Huber \cite{Huber}, M\"uller-\v{S}ver\'ak \cite{Muller1995} and Kuwert-Li \cite{Kuwert2010} we can assume that every branched Willmore immersion is conformally parametrised away from $\mc{S}$.
\end{remark}
Next we recall a generalized Gauss-Bonnet and inversion formula for general surfaces with branch points and ends.
\begin{theorem}[{\cite[Theorem 4.1, 4.2 and Corollary 4.3, 4.4]{Nguyen2011}}]\label{thm_three}
Let $ f:\Sigma\to  \R^3$ be a branched conformal immersion of a Riemann surface $\Sigma$. We denote by $ E = \{ a_{1}, \dots , a_{b}\}$, $b\in \N_0$, the complete ends with multiplicity $k(a_i)+1$, $1\le i\le b$, and for each 
$p\in \Sigma \backslash E$ we denote by 
$m(p)\in \N$ the multiplicity of the point $p$. For $x_0\in \R^3$ we let $\tilde{f}=I_ {x_0} \circ f:=x_0+\frac{f-x_0}{|f-x_0|^2}:\Sigma\backslash f^{-1}(x_0) \to \R^3$ and 
we denote by $\tilde{\Sigma}=\tilde{f}(\Sigma \backslash f^{-1}(x_0))$ the image surface of $\tilde{f}$. Then we have the formulas
\begin{align}
 \int_\Sigma K d\mu =& 2\pi \Big( \chi_\Sigma-\sum_{i=1}^b (k(a_i)+1)+\sum_{p\in \Sigma\backslash E} (m(p)-1)\Big),\label{gauss-bonnet}\\
\int_{\widetilde{\Sigma}} \widetilde{K} d\widetilde{\mu}=&\int_\Sigma K d\mu +4\pi \Big(  \sum_{i=1}^b (k(a_i)+1)-\sum_{p\in f^{-1}(x_0)} m(p) \Big), \label{inversion1}\\
 \mc W ( \tilde f ) =& \mc W (f) + 4\pi \Big(\sum_{i=1} ^b ( k(a_i)+1) -  \sum_{p \in f ^ { -1} ( x_0)} m(p) \Big)\ \ \ \text{and} \label{inversion2}\\
\int_{\widetilde{\Sigma} }|\widetilde A|^{2}d \widetilde \mu   =&  \int_{\Sigma} |A|^2 d\mu + 8\pi \Big(\sum_{i=1} ^b ( k(a_i)+1) - \sum_{p \in f ^ { -1} (x_0)} m(p) \Big).\label{inversion3} 
\end{align} 
\end{theorem}
The last two result we need in the sequel are two classification results for minimal surfaces.
The first one is due to Osserman and gives a classification of complete minimal surfaces with finite total curvature equal to $8 \pi$. 
\begin{theorem}[{\cite[Theorem 3.4]{Osserman1964}}]\label{thm_osserman}\label{thm_hoss}
Let $ \Sigma\subset \R^3$ be a complete minimal surface with $ \int_{\Sigma} |A|^{2} = 8\pi$. Then $\Sigma$ is either isometric to Enneper's minimal surface or a catenoid.  
\end{theorem}
The second theorem classifies complete minimal surfaces with total bounded curvature equal to $ 16 \pi$ and conformal to a sphere $ \mbb{S}^ 2 $ with three points removed. The map $ g(z)$ is up to a biholomorphism the Gauss map and $ \omega$ is a $1$-form, both of which appear the Weierstrass-Enneper representation of the surface.
\begin{theorem}[{\cite[Theorem 3]{L'opez1992}}]\label{thm_lopez}
Let $ \Sigma\subset \R ^ 3 $ be a complete minimal surface with finite total curvature $ \int _{ \Sigma} |A|^ 2 d\mu = 16 \pi$ and suppose that  $\Sigma$ is conformal to $ \mbb{S}^2 \backslash \{ p_1, p_2 , p_3\}$ then up to homothety and rigidity in $\R^3$, 
 \begin{align*}
\Sigma & = \mbb{C} \backslash \{ \frac{1}{ \sqrt{ 3} } , - \frac{ 1 }{ \sqrt{3} } \} \\
g (z) &= B \frac{ z ^ 2 + c z + d } { z +a} \\
\omega &= \theta \frac{ (z +a) ^ 2 } { (z ^ 2 - \frac13 ) ^ 2 } d z 
\end{align*}
where 
\begin{enumerate} 
\item if $ a \neq \frac{ 1 }{ \sqrt { 3 }} , - \frac 1{ \sqrt { 3}} $ given $ r _1, r_2 \in \R, r _ 2 \neq 0$  then
\begin{align*}
c =0, \quad 12 a ^ 4 - ( r_2 ^2 + 3 r_1 ^ 2 + 4 ) a ^ 2 - r_1 ^2 =0, \quad a ^ 2 ( 1 -3d) ^ 2 = r _1 ^2 \\
\theta = 1, \quad B^ 2 = \frac{ 3 | 3a^2 -1|^2 } { r_2^2}
\end{align*} 
\item or $$ c= 0 , \quad a = \frac 1 {\sqrt{3}}, - \frac 1 {\sqrt{3}}, \quad d = 1, \quad \theta =1, \quad B \in \R \backslash \{0\}.$$  
\end{enumerate}

\end{theorem}

\section{Willmore Spheres with Branch Points}\label{sec_willmore_spheres}

In this section we consider branched Willmore immersions $f:\Sigma \to \R^3$ with $\int_\Sigma |A|^2d\mu <\infty$ and we let $Y:\Sigma \backslash \mc{S} \to Q^4$ be the conformal Gauss map associated to $f$ (see section A.3 for the definition). Here $\mc{S}=\{p_1,\ldots,p_l\}$, $l\in \N_0$, is the at most finite set of finite area branch points of $f$. Our goal is to show that the $4$-form $ \la \alpha, \alpha \ra^{ (4,0)}=\la Y_{zz},Y_{zz} \ra dz^4$ is meromorphic with poles of order at most $2$.
 
In a first step we use Theorem \ref{thm_expansion} in order to show that the conformal Gauss map $Y$ associated to a branched Willmore immersion $f$ as above can be estimated close to a branch point by $|\la Y_{zz} , Y_{zz} \ra| \leq C| z| ^ { -2 - \e } $ for any $\e>0$. Analysing the Laurent expansion then yields that the meromorphic $4$-form $\la Y_{zz},Y_{zz} \ra dz^4$ has at worst double poles. 
\begin{theorem}\label{poles}
Let $f:\Sigma \to \R^3$ be a branched Willmore immersion and let $Y$ be the conformal Gauss map associated to $f$. Then the $4$-form $ \la \alpha, \alpha \ra^{ (4,0)}$ is meromorphic with poles of order at most $2$.
\end{theorem}
\begin{proof} 
The immersion $ f$ is a smooth Willmore immersion away from the at most finite set $\mc{S}=\{p_1,\ldots,p_l\}$ and hence we know from the results in section A.3 of the appendix that $\la \alpha,\alpha\ra^{(4,0)}$ is holomorphic away from $\mc{S}$. Therefore it remains to study the quartic form close to the finitely many points in $\mc{S}$. 
Since this is a local problem and using the results from section $2$ we assume from now on that we only have one finite area branch point $p$ with multiplicity $m(p)=m\in \N$, that there exists a conformal parametrization $f:D_1\backslash \{0\} \to \R^3\backslash \{0\}$ of $\Sigma$ and that all the results from Theorem \ref{thm_expansion} are applicable.

Next we recall the formula \eqref{eqn_main} from the appendix
\begin{align}
\la Y _ {zz} , Y _ { zz} \ra = H _ { zz} \varphi + H^2 \varphi^2 / 4 - H_{z} ( \varphi e ^ { - \lambda } ) _ { z} e ^{\lambda}=:I+II+III,
\end{align} 
where $\lambda=2u$ and $\varphi$ is the Hopf differential. 

Now we know by Theorem \ref{thm_expansion} that for all $z\in D_1\backslash \{0\}=:D_1^\star$ and all $k\in \N_0$, $\e>0$ we have
\begin{align*}
|\nabla^k A|(z) \le C | z|^{1-k-m-\e}.
\end{align*}
Furthermore we note that 
 \begin{align}
| \varphi | (z) = | \circo A |(z) e ^ {\lambda(z) }\le C|z|^{m-1-\e} \label{estvarphi} 
\end{align} 
where we used \eqref{expansionu}. Hence we can estimate for $\e<\frac14$
\begin{align}
 |I|+|II|\le C (|z|^{-1-m-\e} |z|^{m-1-\e}+|z|^{2-2m-2\e}|z|^{2m-2-2\e})\le C|z|^{-2-2\e}. \label{I+II}
\end{align}
Therefore it remains to estimate
\begin{align*}
III=H_z ( \varphi e ^ { -\lambda} ) _ z e ^ {  \lambda} = H _z \varphi_z  + H_{z}\varphi  \lambda_z=:III_a+III_b.
\end{align*}

We start with the term $III_b$. Using the facts that $|A(z)|\le C|z|^{1-m-\e}$ and $e^{2u(z)}\le C|z|^{2m-2}$ for all $z\in D_1^\star$ we conclude that
\[
 K_ge^{2u} \in L^p(D_1)\ \ \ \forall \ \ 1\le p<\infty.
\]
From \eqref{pseudoliouville} and standard $L^p$-theory we then get that $v\in W^{2,p}_{\text{loc}}(D_1)$ for all $1\le p<\infty$. Therefore we can improve \eqref{expansionu} to get for all $z\in D^\star_{\frac12}$
\begin{align}
 u(z)=(m-1)\log |z| +w(z), \label{expansionu2}
\end{align}
where $w\in C^{1,\alpha}(D_{\frac12})$ for every $0<\alpha<1$.

In particular we conclude that for every $z\in D^\star_{\frac12}$ we have
\begin{align*}
 |\nabla \lambda(z)| \le \begin{cases} C &\mbox{if } m=1 \\
                      Cm|z|^{-1} &\mbox{if } m\ge 2.
                     \end{cases}
\end{align*}
Combining this with the estimates for $|H_z|$ and $|\varphi|$ we get
\begin{align}
 |III_b|\le C|z|^{-m-\e}|z|^{-1}|z|^{m-1-\e}\le C|z|^{-2-2\e}.\label{IIIb}
\end{align}

In order to estimate the term $III_a$ we need an estimate for $\varphi_z$. We claim that for any $z\in D^\star_{\frac14}$ we have
\begin{align}
 |\varphi_z(z)|\le C|z|^{m-2-\e}. \label{claim}
\end{align}
Combining this again with the estimate for $|H_z|$ we then conclude
\begin{align}
 |III_a|\le C|z|^{-m-\e}|z|^{m-2-\e} \le C|z|^{-2-2\e}.\label{IIIa}
\end{align}
The estimates \eqref{I+II}, \eqref{IIIb} and \eqref{IIIa}, together with the Laurent expansion, give the desired result for $\e$ small enough.

In order to prove \eqref{claim} we use the Codazzi equation in complex form, which is given by 
\begin{align*}
\varphi_{ \ov z} = e ^ \lambda H_z\ \ \ \text{on}\ \ D_1^\star.  
\end{align*}   
Since $\varphi \in L^p_{\text{loc}}(D_1)$ for all $1\le p<\infty$ and $\varphi_{\ov z} \in L^q_{\text{loc}}(D_1)$ for all $1\le q<2$, the Codazzi equation is actually satisfied weakly on all of $D_1$.

For technical reasons we introduce a cut-off function $\eta \in C^\infty_c(D_1)$ with $0\le \eta \le 1$, $\eta\equiv 1$ on $D_{\frac12}$ and $\eta \equiv 0$ on $D_1\backslash D_{\frac34}$. Note that that we can choose $\eta$ in such a way that $||\nabla^k \eta||_{L^\infty(D_1)}\le c$ for all $k\in \N$.
Now we define $\psi=\eta \varphi$ and we note that $\psi$ satisfies
\begin{align}
 \psi_{\ov z}=\eta e^\lambda H_z+\eta_{\ov z} \varphi =:v. \label{eqpsi}
\end{align}
Using the previous results we get the estimate
\begin{align}
|v(z)|\le C |z|^{m-2-\e}. \label{estv}
\end{align}

The solution of equation \eqref{eqpsi} is for every $z\in \C$ given by 
\begin{align*}
\psi(z)=\mathcal{C}(v) (z)+h(z),
\end{align*}
where $h$ is some holomorphic function and 
\[
 \mathcal{C}(v) (z)= - \frac { 1 }{ \pi } \int_{ \C} \frac{v(\xi)} { \xi - z    } d \xi
\]
is the Cauchy transform of $v$ (see e.g. \cite{astala}, Chapter 4). Since the holomorphic function is smooth at the origin we assume without loss of generality that $h\equiv 0$. Note that the operator $v\mapsto \psi$ maps $L^p$ into $W^{1,p}$ for any $1<p<\infty$ and hence we conclude that $\psi\in W^{1,q}(D_1)$ for all $1\le q<2$.
Since $v\in L^p(\C)$ for every $1<p<2$ we know that $\mathcal{C}(v)\in L^{q}(\C)$ for every $2<q<\infty$.

Now we calculate $ \frac{ 1}{ \xi ( \xi - z ) } = -\frac{1}{ z} \frac{1}{\xi} + \frac {1}{z} \frac{ 1 }{ \xi - z } $ and hence we conclude 
\begin{align}
\psi(z)=\mathcal{C}(v)(z) &= -\frac{1}{ \pi} \int_{\C} \frac{ \xi v(\xi)}{ \xi( \xi-z)} d\xi \nonumber \\
& = \frac{1}{ \pi z}\int _{\C}\frac {\xi v(\xi)}{\xi} d \xi -  \frac{ 1}{ \pi z}\int_{\C} \frac { \xi v(\xi) }  { \xi - z } d \xi \nonumber \\
& = \frac{1}{z} ( I ( z) - I ( 0) ) ,\label{reppsi}
\end{align} 
where $  I(z) =\mathcal{C}(\tilde{v})(z)$ and $\tilde{v}(w)=wv(w)$ for all $w\in \C$.
Using \eqref{estvarphi} we get that
\begin{align}
 |I(z)-I(0)|\le C|z| |\psi(z)|\le C|z|^{m-\e}. \label{estI}
\end{align}
Moreover we have that $|\tilde{v}(w)|\le C\chi_{D_1} |w|^{m-1-\e} \in L^p(\C)$ for all $1\le p <\infty$ and hence $I(z)-I(0)$ satisfies for all $z\in \C$
\begin{align}
 \Big(I(z)-I(0) \Big)_{\ov z} = \tilde{v}(z). \label{eqI}
\end{align}

Note that this is again a Cauchy-Riemann equation, but compared to the Codazzi equation, we improved the decay of the right hand side by the order of one. 

Next we differentiate \eqref{eqI} with respect to $z$ in order to get 
\begin{align}
 \Delta \Big(I(z)-I(0) \Big)= \tilde{v}_z(z)=v(z)+z v_z(z).\label{laplaceI}
\end{align}

In the following we consider two cases.
\newline
\underline{Case 1:} $m\ge 2$

First we note that for $z\in D^\star_{\frac12}$ we have
\[
 v(z)+zv_z(z)=e^{\lambda(z)} \Big( H_z(z)+z \lambda_z (z) H_z(z)+zH_{zz}(z) \Big) 
\]
and by using the previous decay estimates we conclude for all $z\in D^\star_{\frac12}$ 
\begin{align*}
 |\Delta \Big(I(z)-I(0) \Big)| \le C|z|^{m-2-\e} = C|z|^{(m-1)-2+\beta},
\end{align*}
where $0<\beta=1-\e<1$. Moreover we estimate for $0<\varrho<\frac12$
\[
 \int_{D_\varrho} |z|^{-(m-1)} |\Delta \Big(I(z)-I(0) \Big)|\,\ dz \le C\int_{D_\varrho} |z|^{-1-\e}\,\ dz=C_\e \varrho^\beta.
\]

The last two estimates allow us to apply Lemma 2.1 of \cite{kuwert07} (which is an extension of Lemma 1.1 in \cite{caffarelli}) in order to get for every $z\in D_{\frac14}^\star$
\begin{align*}
 I(z)-I(0) =& P(z) +O(|z|^{m-\e})\ \ \ \text{and}\ \ \
 \nabla I(z)= \nabla P(z) +O(|z|^{m-1-\e}),
\end{align*}
where $P$ is harmonic polynomial of degree at most $m-1$. Combining this expansion with \eqref{estI} we get that $P\equiv 0$ and hence we have for every $z\in D^\star_{\frac14}$
\[
 |I_z(z)|\le C|z|^{m-1-\e}.
\]

Differentiating \eqref{reppsi} with respect to $z$ we get for any $z\in D^\star_{\frac14}$
\[
 |\varphi_z(z)| =|\psi_z(z)|= |-z^{-2} \Big( I(z)-I(0) \Big) +z^{-1} I_z(z)|\le C|z|^{m-2-\e}.
\]
\underline{Case 2:} $m=1$

In this case we use \eqref{expansionH} in order to conclude that for $z\in D^\star_{\frac12}$ we have $H_z(z)=\frac{H_0}{2} z^{-1} + \omega_z(z)$, where $\omega \in W^{2,p}(D_{\frac12})$ for any $1\le p<\infty$, and therefore
\[
 \big( zv(z) \big)_z=\big( e^{\lambda(z)} (\frac{H_0}{2}+z\omega_z(z)) \big)_z=e^{\lambda(z)}\big( \lambda_z(z) ( \frac{H_0}{2}+z\omega_z(z))+\omega_z(z)+z\omega_{zz}(z) \big).
\]
Since $e^\lambda$ and $\lambda_z$ are bounded, we conclude that $ \big( zv(z) \big)_z\in L^p(D_{\frac12})$ for all $1\le p<\infty$. Standard $L^p$-theory applied to \eqref{laplaceI} and the Sobolev embedding theorem therefore imply that 
\[
 I \in W^{2,p}_{\text{loc}}\cap C^{1,\alpha}_{\text{loc}}(D_{\frac12})\ \ \ \forall \,\ 1\le p< \infty\ \ \ \text{and} \,\ \forall \,\ 0<\alpha <1.
\]
Differentiating \eqref{reppsi} again with respect to $z$ we get for any $z\in D^\star_{\frac14}$
\[
 |\varphi_z(z)| =|\psi_z(z)|= |-z^{-2} \Big( I(z)-I(0) \Big) +z^{-1} I_z(z)|\le C|z|^{-1}.
\]
\end{proof}
We note that the previous proof yields a slight improvement of Theorem \ref{thm_expansion} in the case $m=1$.
\begin{corollary}\label{improve}
Let $f$ and $\Sigma$ be as in Theorem \ref{thm_expansion} with $m=1$. Then for all $0< \alpha < 1$ we have the expansion
\begin{align*}
 |A|(z)=A_0 |\log |z||+C^{0,\alpha}_{\text{loc}},
\end{align*}
where $A_0\in \R^+$. 
\end{corollary}
\begin{proof}
The Gauss equation yields
\[
 |A|^2=|\circo A|^2+\frac12 H^2.
\]
In the proof of Theorem \ref{poles} for $m=1$ we showed that the function $I$ is in $C^{1,\alpha}_{\text{loc}}(D_{\frac12})$ and hence, using \eqref{reppsi} and the fact that $\varphi=\psi$ in $D_{\frac12}$, we get
\[
 \varphi \in C^{0,\alpha}_{\text{loc}}(D_{\frac14}).
\]
Combining this with \eqref{estvarphi} and \eqref{expansionu2} we get
\[
 |\circo A| \in C^{0,\alpha}_{\text{loc}}(D_{\frac14}).
\]
Therefore the claim follows from \eqref{expansionH}.
\end{proof}

The following theorem is a modification of a result of Bryant.
\begin{theorem}\label{thm_mobius}
Let $ f : \Sigma \rightarrow \R^3$ be a branched Willmore immersion, let $ Y: \Sigma \backslash \mc{S} \rightarrow Q^4$ be its conformal Gauss map and suppose that $ \la \alpha, \alpha \ra ^ { (4,0)} \equiv 0$ vanishes identically on $\Sigma$. Then $ f$ is either umbilic away from $\mc{S}$ or a M\"obius transform of a complete branched minimal immersion $ f' :\Sigma\backslash \mc {S} \rightarrow \R^ 3 \cup \{\infty \}$.
\end{theorem}    
\begin{proof}
In the proof of this Theorem we use several results and notations from the appendix. 

Let $ f : \Sigma\rightarrow \R^3$ be a branched Willmore immersion, then the conformal Gauss map $ Y: \Sigma\backslash \mc{S} \rightarrow Q ^ 4 $ is a conformal harmonic map. If $Y \not \equiv $ constant on $\Sigma \backslash \mc{S}$, that is $f$ is not an umbilic surface, then $ Y$ is a conformal minimal immersion on some open dense subset $ \Sigma ' \subset \Sigma$ where $ \Sigma \backslash \Sigma'$ is the union of the set of umbilic points and $\mc{S}$. Now by hypothesis the second fundamental form of $Y$ satisfies $\langle \alpha, \alpha \rangle^{ (4,0)} \equiv 0$ on all of $\Sigma$. Hence on the open dense set $\Sigma'$ the induced metric has signature $ (+,-)$ so it contains two null lines $ [N_1]$ and $ [N_2]$. As mentioned in the appendix, each null line corresponds to a sphere congruence. In particular, one null line $ [N_1]$ corresponds to the Willmore immersion $ f : \Sigma'\rightarrow \R^3$ and the other null line $ [N_2]$ corresponds to Bryant's conformal transform $ \hat f : \Sigma' \rightarrow \R^ 3 $. By \eqref{prop_hol} and the discussion following it (which is a local result and hence applies to $\Sigma'$) the conformal transform is a constant line. As the surface is smooth on umbilic points away from the branch points, the conformal transform is a constant line on $ \Sigma\backslash \mc {S}$.  

Hence $ \hat { f} \in \R ^ 3 \cup \{\infty \}$ is a point and furthermore we claim that $ \hat f \in f (\Sigma)$. Suppose not, then by a suitable M\"obius transform, map the point $ \hat { f} $  to $ \{ \infty \}$. Now since every mean curvature sphere of $f$ must touch $\hat {f}$, this shows that each mean curvature sphere of $f$ is a plane, which implies that the mean curvature sphere has mean curvature zero. Since any surface has the same mean curvature as its mean curvature sphere, it must be a minimal surface. But if $ \hat {f} \not \in f(\Sigma) $, then $ I_{ \hat {f}} \circ f (\Sigma)$ is a compact branched minimal surface, which is a contradiction (see e.g. \eqref{inversion2}). Hence $ \hat { f} \in f (\Sigma)$. Furthermore, the argument above shows that after M\"obius inversion at $ \hat f $, $ I_{ \hat f } \circ f $ is a branched minimal immersion with finite total curvature.
\end{proof}
We note that branch points of $f$ do not necessarily lie in $\hat f$, hence after inversion there may be branch points in the interior of the minimal surface.  
  
To apply the above theorem, we use the Riemann-Roch theorem in order to compute the dimension of the space of meromorphic forms with at worst double poles. 
\begin{theorem}[Riemann-Roch for Line Bundles, \cite{Gunning1966}]\label{thm_RiemannRoch}
Let $\xi$ be a complex line bundle over a Riemann surface $M$ of genus $g$. Then we have
\begin{align*}
 \dim( \Gamma(M,\mathcal{O}( \xi))) - \dim ( \Gamma( M,\mathcal{O}( \kappa \cdot\xi^{-1}))) = c_{1}(\xi) + 1- g 
\end{align*}
where $c_1(\xi)$ is the first Chern class of $\xi$ and $\kappa$ is the canonical line bundle. Let us write $\gamma(M,\mathcal{O}(\xi))=\dim \Gamma ( M, \mc{O}(\xi))$ for the dimension of holomorphic cross sections. Then we have :
\begin{align*}
 \begin{array}{|l|l|}
 \hline
 c_1 ( \xi) < 0 & \gamma(\xi) = 0 \\
 c_1(\xi) =0& \gamma(\xi) = \left \{ 
 				\begin{array}{ll}
				1 &\text{ if $\xi =1$} \\
				0 &\text{ if $\xi \neq 1$} 
				\end{array}
				\right. \\
c_1(\xi) = 2g -2 & \gamma(\xi) = \left\{ 
						\begin{array}{ll}
						g &\text { if $\xi= \kappa$},\\
						g-1 &\text{ if $ \xi \neq \kappa$}\\  
						\end{array}
						\right  .\\
c_1(\xi) > 2g-2 & \gamma(\xi) = c_1 ( \xi) - ( g-1) 	\\
\hline						
 \end{array}
\end{align*}
\end{theorem}

Let $\Sigma$ be a branched Willmore surface with branch points $p_i$ of multiplicity $m_i$. The next theorem shows that there are non nontrivial meromorphic four forms with poles of order at worst $2$ lying on $D=\sum_i m_i p_i$. 
   
\begin{proposition} \label{prop_poles}
Let $f : \mbb{S}^{2} \rightarrow \R^3 $ be a branched Willmore immersion with divisor $D, |D| \leq 3$ then there are no non-trivial meromorphic four forms with at worst double poles lying on $D$.    
\end{proposition}
\begin{proof}
Let $D$ denote the divisor, then being a meromorphic four form with at worst double poles on $D$ is equivalent to being a holomorphic four form on the bundle $ \xi = \kappa^4 \otimes 2 D$, where $ \kappa$ is the anti-canonical line bundle which in particular has $ c_1( \kappa) = 2( g-1)$.   

By applying the Riemann-Roch theorem to the bundle $ \xi = \kappa^4\otimes 2D $ where $ D = \sum _{ i = 1} ^ d m _ i z _i $ is a divisor, then as $ \gamma ( \xi) = c_1 ( \xi) - (g-1)$ we get that 
\begin{align*}
\gamma ( \kappa ^ 4 \otimes 2D ) &=  c_1 ( \kappa ^ 4 \otimes 2D ) - (g-1) \\
& = 4c_1 ( \kappa) + c_1 (2D) - (g-1)\\
& =  2| D| + 7 ( g-1).
\end{align*} 
Hence we see that as $g=0$ and that $|D| \leq 3$ shows that $ \gamma ( \kappa^4 \otimes 2 D ) = 0$ and that the only such meromorphic four form is the trivial form.  
\end{proof}

Furthermore we state the following fact, 
\begin{lemma} \label{lem_odd}
Let $ f:\mbb{S}^2 \to \R^3 $ be a branched conformal immersion with exactly one finite area branch point $p\in \mbb{S}^2$. Then the multiplicity $m(p)$ can not be even.     
\end{lemma}
\begin{proof} 
Suppose that $p\in \mbb{S}^2$ is the only finite area branch point and that $m(p)$ is even.  
Consider the inversion at the point $f ( p) $, $ I_{f ( p)} \circ f $. This then gives us a complete branched conformal immersion $\tilde{f}:\tilde{\Sigma}=\tilde{f}(\mbb{S}^2\backslash f^{-1}(f(p)))\to \R^3$ with $\#\{ f ^{-1} ( f(p))\} $ ends of which exactly one is an end of odd branch order $ k(p)=m(p)-1$ and all the other ends have branch order zero. Then we may compute by the Gauss-Bonnet formula \eqref{gauss-bonnet} applied to the complete surface $\tilde{\Sigma}$ with finite total curvature     
\begin{align*}
\int_{\tilde{\Sigma}} \tilde{K} d \mu _{\tilde{g}}  &= 2 \pi ( \chi (\tilde{ \Sigma}) - \sum_{ q \in f^{-1}( f ( p )) }   (k(q )+1))  \\
& = 2 \pi( 2 - 2\# \{ q \in f ^{-1}( f ( p ) ), q \neq p \} -1 - (k(p)+1)) \\
& = 4\pi l-2\pi k(p)
\end{align*}
for some $l\in \Z$. However by a theorem of White \cite{White1987} such a surface must satisfy $ \int_{\tilde{\Sigma} } \tilde{K} d \mu_{ \tilde{g}  } = 4 \pi m$ for some $m\in \Z$ and this is a contradiction.
 \end{proof}

An application of Proposition \ref{prop_poles} and Theorem \ref{thm_mobius} immediately gives us the following theorem.
\begin{thm} \label{thm_minimal}
Let $f: \mbb{S}^2 \to \R ^ 3 $ be a branched Willmore immersion with $|D| \leq 3$, then 
$f$ is either umbilic or it is the M\"obius transform of a branched complete minimal surface. In particular, this implies that the Willmore energy is quantised as $\mc {W} ( f) = 4 \pi k$.  
  \end{thm} 
  \begin{proof} 
We have shown that if the surface is not umbilic away from the branch points then it must be the M\"obius transform of a complete branched minimal surface. Hence we only have to consider the case of an umbilic surface away from the branch points. Away from the branch points the surface is then a round sphere. 

First we claim that all multiplicity one branch points $p\in \mbb{S}^2$ are smooth. This follows from the fact that we can invert the surface at an arbitrary point $f(q)\not= f(p)$ and as a result we obtain a complete, branched minimal surface. The multiplicity one branch point $p$ then corresponds to an interior branch point of multiplicity one of the minimal surface. But these branch points are removable and after composing with the same inversion again this shows the claim.

Hence we only have to consider the cases $D=2p$ and $D=3p$ for some $p\in \mbb{S}^2$. The case $D=2p$ is ruled out by Lemma \ref{lem_odd}. In order to rule out the case $D=3p$ we invert the surface at $ f( p)$ and as a result we obtain a flat, complete surface $\tilde{\Sigma}$ with an end of branch order two and no interior branch points. This contradicts \eqref{gauss-bonnet} since
\[
0=\int_{\tilde{\Sigma}} \tilde{K} d\mu_{\tilde{g}}=2\pi(1-3)=-4\pi.
\]
As we only consider connected surfaces, the umbilic sphere is one copy of a round sphere. 

The fact that the Willmore energy is quantised now follows from \eqref{inversion2}.
\end{proof}  
 
Using the classification results for minimal surfaces mentioned in section 2 and the Weierstrass-Enneper representation of complete, branched minimal surfaces (see Appendix B), we can prove the following theorem.
 \begin{thm3}
Let $f: \mbb{S}^2 \rightarrow \R ^ 3 $ be a branched Willmore immersion with $|D| \leq 3$ and $ \mc {W}(f) < 16 \pi$. Then one of the following cases occurs ($p_i\in \mbb{S}^2$, $1\le i \le 3$): 
\begin{enumerate}
\item $D = \emptyset$ and $f$ is an isometric immersion of the round sphere with $ \mc{W} (f) = 4 \pi$.

\item $D =p_1+ p_2$ and $ f$ is the M\"obius transform of an embedding of a catenoid with $\mc{W} ( f) = 8\pi$.

\item $D = 3p_1$ and $f$ is the M\"obius transform of an immersion of Enneper's minimal surface with $\mc{W}(f) = 12\pi.$

\item $ D =p_1+ m_2p_2+ m_3p _3 , m _i \in\{0,1\} $ and $ f$ is the M\"obius transform of a trinoid with $ \mc{W}(f) = 12 \pi$.     

\end{enumerate}

\end{thm3}
\begin{proof}
Theorem \ref{thm_mobius} and Theorem \ref{thm_main} show that $f$ is either a round sphere or the M\"obius inversion of a branched, complete minimal surface where the inversion point is the conformal transform, $ \hat f$. We will denote by $ \hat x  $ the set of points whose image is in $ \hat f$ that is $ \hat {x} = f^{-1} ( \hat f )$, and we let $\hat {\Sigma} = \Sigma \backslash \mc {S}$. The image of a branch point may not lie in the conformal transform, hence the resulting complete minimal surface may have interior branch points but of course if the conformal transform contains a branch point then it must lie in the divisor. 

\underline{Case 1}: $\mc {W}(f) = 4\pi$.

In this case we have an umbilic sphere and from \eqref{inversion2} we conclude that each branch point has at most multiplicity one. As we already noted in the proof of Theorem \ref{thm_main}, multiplicity one branch points are smooth. Therefore $D=\emptyset$.

\underline{Case 2}: $\mc {W}(f) = 8\pi$. 

After inverting the surface at $\hat{f}$ we get from \eqref{inversion2} that $\hat{x}=2p_1$ or $\hat{x}=p_1+p_2$. Moreover, inverting the surface at the image of every possible branch point, we conclude again from \eqref{inversion2} that any branch point has at most multiplicity two. Note that $D\not= \emptyset$ by the classification result of Bryant \cite{Bryant1984}. Moreover, arguing as in the proof of Theorem \ref{thm_main}, we get that multiplicity one branch points $q\notin \hat{x}$ correspond after inversion at $\hat{f}$, to interior branch points of multiplicity one of a minimal immersion and hence they are smooth. 
\begin{enumerate} 
\item $ \hat {x} = 2 p_1$
\newline
By the above discussion we must have $ D = 2 p_1 $ (note that $|D|\le 3$) but this is ruled out by Lemma \ref{lem_odd}.

\item $ \hat { x} = p_1+ p_2$. 
\newline
We can then have $D = p_1 + p _2 $ or $ D = 2q$ for some $q\in \mbb{S}^2\backslash \{p_1,p_2\}$ (otherwise we get a contradiction to \eqref{inversion2}). 
\newline
If $D = p _ 1 + p _2 $, then by applying the Gauss-Bonnet formula to $\tilde{f}=I_{\hat{f}} \circ f$, we conclude from \eqref{gauss-bonnet} 
\begin{align*}
\int_{ \tilde{\Sigma}} \tilde{K} d \mu_{\tilde{g} } &= 2\pi \bigg( \chi ( \tilde{\Sigma}) - \sum_{p_i \in \hat x}  (k ( p_ i )+1) \bigg) \\
   & = - 4 \pi.
\end{align*} 
Hence $ \int_{\tilde{\Sigma}} |\tilde{A}|^2 d\mu_{\tilde{g}}=-2\int_{\tilde{\Sigma}}\tilde{K}d\mu_{\tilde{g}}= 8 \pi$ and by Theorem \ref{thm_osserman} this implies that $\tilde{f}$ is the embedding of a catenoid.
\newline
If $ D = 2 q$ then inverting the surface at $ f( q)$ gives a minimal surface by \eqref{inversion2} since $\mc{W} (f) = 8 \pi$. As $  p_1, p_2 $ then correspond to interior branch points of multiplicity one they are smooth, and this yields a contradiction by Lemma \ref{lem_odd}.
\end{enumerate}

\underline{Case 3} : $ \mc {W} ( f) = 12 \pi$.

In this case $\hat{x}$ has multiplicity three and any branch point has maximal multiplicity three by the same arguments as in Case $2$. The conformal transform then may have the form $ \hat x = 3p_1, \hat x = 2 p _1 + p_2, \hat x = p_1+p_2+p_3$. As before we can not have branch points $q\notin \hat{x}$ of multiplicity one. 
\begin{enumerate} 
\item $\hat x= 3 p_1$ 
\newline
Then $D= 3 p_1$ and after inverting the surface at $\hat{f}$ and noting that the resulting surface has no interior branch points, we get from \eqref{gauss-bonnet} as above
\begin{align*}
\int_{ \tilde{\Sigma}} \tilde{K} d \mu _{ \tilde{g} } &= 2 \pi ( \chi( \tilde{\Sigma}) - (k(p_1)+1)) \\
 &= - 4 \pi.     
\end{align*} 
Hence $ \int_{\tilde{\Sigma}} |\tilde{A}|^2 d\mu_{\tilde{g}}= 8 \pi$ and by Theorem \ref{thm_osserman} $ I _{ \hat f} \circ f $ is an immersion of Enneper's minimal surface.  

\item $ \hat x = 2 p_1 + p_2 $ 
\newline
In this case $ D = 2 p_1 + mq$ for $m = 0,1$ and some $q\in \mbb{S}^2\backslash \{p_1\}$. If $ m =1 $ then $q=p_2$, as we can not have an interior branch point $q\notin \hat{x}$ of multiplicity one. Hence $D = \hat x $ and therefore the complete minimal immersion $\tilde{f}=I_{\hat{f}} \circ f$  is unbranched. \eqref{gauss-bonnet} then yields 
\begin{align*}
\int_{ \tilde{\Sigma} }\tilde{K}  d \mu_{\tilde{g}} &= 2 \pi ( \chi ( \tilde{\Sigma}) - (k(p_1)+1) - (k ( p_2)+1) )\\
& = 2\pi ( 2 -2 -2 - 1 ) = - 6\pi,
\end{align*}
contradicting the fact that $\int_{ \tilde{\Sigma} }\tilde{K}  d \mu_{\tilde{g}}=4\pi k$, for some $k\in \Z$, by a theorem of White \cite{White1987}.
\newline
The case $ D = 2 p_1$ (i.e. $m=0$) is again ruled out by Lemma \ref{lem_odd}.

\item $ \hat x = p_1 + p_ 2 + p_3 $. 
\begin{enumerate} 
\item $ D = 3q $ for some $q\in \mbb{S}^2\backslash \{p_1,p_2,p_3\}$. Using \eqref{inversion2} it follows that after inverting the surface at $ f ( q )$, $ I _{ f (q)} \circ f $ is a minimal immersion with an end of branch order $2$. Using again the fact that interior branch points of a minimal surface of multiplicity one are smooth, we can then repeat the argument of Case $2$, (1), in order to show that $I _{ f (q)} \circ f $ must be an immersion of Enneper`s minimal surface. This fact is already sufficient for our purposes. 

Using the Weierstrass-Enneper representation (see Appendix B) we can actually show that this case can`t occur at all. Indeed, the above reasoning shows that the conformal transform $ \hat x $ consists of three interior multiplicity one branch points, hence they must be smooth. In particular this shows that the $p_i$ correspond to planar ends for $\tilde{f}=I_{\hat f} \circ f$ and $\tilde{f}$ has an interior branch point of multiplicity three. We then have that 
\begin{align*}
\partial ( I _{ \hat f } \circ f ) = \left ( 
\begin{array}{c} 
\phi_1 \\
\phi_2\\
\phi_3
\end{array}
\right)
\end{align*}    
where $ \phi_i$ are linearly independent meromorphic $1$ forms which have poles of order $ 2$ at each end and zero residue at each end. Furthermore, the $ \phi_i$ have a zero of order two on $D$. This is equivalent to saying that $\phi_i$ are meromorphic $1$-forms where $ \tilde H^0 = \{ \gamma \in H^{0} (K _ \Sigma \otimes [ \hat x] \otimes [ - 2 q ] ) \mid res_{p_i} \phi_j= 0\}  $. By the Riemann-Roch theorem ,
\begin{align*}
| \tilde H^0 | = 6 -1 -2 -2 = 1.
\end{align*}
Hence there is no such minimal surface.

\item $D = 2 q_1+ m q_2$ for $m = 0,1$ and some $q_1\in \mbb{S}^2 \backslash \{p_1,p_2,p_3\}$, $q_2\in \mbb{S}^2$. Note that the case $m=0$ is ruled out by Lemma \ref{lem_odd}. If $m =1$ we can assume $q_2=p_1$ as there are no interior branch points of multiplicity one. Now inverting the surface at $\hat{f} $ gives a surface with three ends of branch order zero and one interior branch point of multiplicity two. By \eqref{gauss-bonnet} this yields
\[
 \int_{\tilde{\Sigma}} \tilde{K} d\mu_{\tilde{g}}=2\pi(-1-3+1)=-6\pi
\]
which, as we have seen, is impossible.

\item $ D = q_1+ m_2q_2 + m_3q _3 $ for $m_2,m_3=0,1$. As there are no interior branch points of multiplicity one for the surface inverted at $\hat{f}$, the only possibilities are 
\begin{align*}
D = p_1, D= p_1+ p_2\ \ \ \text{or} \ \ D = p_1+ p_2 + p _ 3. 
\end{align*}    
Using \eqref{gauss-bonnet} we calculate
\[
 \int_{\tilde{\Sigma}} \tilde{K}d\mu_{\tilde{g}}=2\pi(2-3-3)=-8\pi
\]
and hence $\int_{\tilde{\Sigma}} |\tilde{A}|^2 d\mu_{\tilde{g}}=16\pi$. These surfaces are trinoids by the classification result of Lopez, see Theorem \ref{thm_lopez}.
\end{enumerate}

\end{enumerate}
\end{proof}

\section{Analysis of the Willmore Flow} \label{sec_analysis}

The Willmore flow is the flow in the direction of the negative $L^2$ gradient of the Willmore functional. We will consider the flow in the smooth setting and analyse the singular behaviour at the first singular time by taking blowups, blowdowns or translations. We will see that Willmore surfaces with finite isolated singularities naturally model singularities of the flow. This contrasts with the mean curvature flow, where the natural models of the singularities of the flow are not minimal surfaces but rather self-shrinkers. This difference can be attributed to the fact that the energy of the Willmore flow is scale invariant and analytically behaves like the harmonic map heat flow from a Riemann surface.  

For geometric flows the short-time existence theory for smooth initial data is standard. The key fact is that although the equation is not strictly parabolic, the zeroes of the symbol of the differential operator are due only to the diffeomorphism invariance of the equation. By breaking this diffeomorphism invariance, the existence theory then reduces to standard parabolic theory.
\begin{theorem}[\cite{Kuwert2011}]
Let $ f _0 : \Sigma\rightarrow \R ^ n$ be a  closed, smoothly immersed surface. Then the initial value problem 
\begin{align*}
 \partial _t f &= - W(f)   \quad \text{ on $\Sigma \times (0,T) $},\\ 
f \mid_{t=0} & = f _ 0
\end{align*}
has a unique smooth solution on a maximal interval $ [0,T)$ where $ 0 < T \leq \infty $.
\end{theorem}

Next we summarize the blow up behaviour of the Willmore flow. Details can be found in section 3.3 of \cite{Kuwert2011}.

Let $0<T\le \infty$ be the maximal time of existence for the Willmore flow of a closed surface with smooth initial data $f_0$ and assume that the flow does not converge smoothly. Then there exist sequences $t_j \nearrow T$, $r_j \in \R^+$ and $x_j\in \R^n$ such that the rescaled flows  
\[
 f_j:\Sigma \times \left[ -\frac{ t_j}{ r_j^4} , \frac{ T - t_j} { r _j ^ 4 } \right) \rightarrow \R^n, f_j( p, \tau) = \frac{ 1}{ r_j} ( f ( p, t_j++ r_j^4 \tau) - x_j)
\]
converge after appropriate reparametrisations locally smoothly to a smooth, non-trivial, properly immersed Willmore surface $\hat{f}_0:\hat{\Sigma}_0 \to \R^n$. One has to distinguish between three different cases: 
\begin{enumerate}
 \item[1)] $r_j\searrow 0$. In this case $\hat{f}_0$ is called a blow up.
 \item[2)] $r_j \to \infty$. Then $\hat{f}_0$ is called a blow down.
 \item[3)] $r_j \to 1$. In this situation $\hat{f}_0$ is called a limit under translations.
\end{enumerate}
The cases $2)$ and $3)$ can only occur for $T=\infty$.

It was shown by Kuwert and Sch\"atzle \cite{Kuwert2001}, Lemma 4.3, that if $\hat{\Sigma}_0$ contains a compact component, then $\hat{\Sigma}_0$ is diffeomorphic to $\Sigma$. Moreover, it was proved in \cite{chill08}, that blow ups and blow downs are never compact.

The following lemma is a slight improvement of Lemma 5.1 in \cite{Kuwert2004} and it gives refined convergence properties of the blow up procedures. 

\begin{lemma}\label{thm_convergence} 
Let $ \Sigma_ j$ be a sequence of smooth and closed surfaces immersed in $\R^3$ which satisfy 
\begin{align}
\label{eqn_willmore_bound}
\mc{W}( \Sigma_j) &\leq C \\
\label{eqn_total_bound}
 \int_{ \Sigma _j } |A_{\Sigma_j} | ^ 2 d \mu_{\Sigma_j} &\leq C \ \ \ \text{and} \\
\label{eqn_local_convergence}
\Sigma_ j &\rightarrow \Sigma \quad \text{ smoothly in compact subsets of $ \R^3 $ } ,
\end{align}
where $ \Sigma$ is a smooth, non-compact Willmore surface. Then for any $ x _0 \not \in \Sigma$ and the inversion $ I_{x_0} (x) = \frac{ x -x _0} { | x-x_0|^{2} }$ we have that $\ov \Sigma \cup \{x_0\}= I_{x_0}(\Sigma) \cup \{x_0\}$ is a closed, branched surface and 
\begin{align}
 \ov{\Sigma} &= I_{x_0} ( \Sigma )  \text{ is a smooth Willmore surface}, \\
\label{eqn_energy_bound} \mc { W } ( \Sigma) + 4 \pi \theta^ 2 ( \mu , \infty)  &= \mc {W}( \ov \Sigma) \leq \liminf_{ j \rightarrow \infty} \mc {W} ( \Sigma_j) \ \ \ \text{and} \\
\label{eqn_genus_bound}  g ( \ov \Sigma)& \leq \liminf _{ j \rightarrow \infty } g (\Sigma_j).  
\end{align}
\end{lemma}
\begin{proof} 
Without loss of generality, we assume that $x_0 = 0$ and we denote $I:=I_0$. Let $\ov \Sigma _ j = I ( \Sigma_j)$. Using the convergence property \eqref{eqn_local_convergence} we see that $ \dist ( 0, \Sigma_j) \rightarrow \dist ( 0, \Sigma) >0$ hence for $j$ sufficiently large $0 \not \in \ov \Sigma _j $. Furthermore, we have that 
\begin{align*}
\ov \Sigma _j \rightarrow\ov \Sigma \quad \text{ smoothly in compact sets of $ \R ^ 3 \backslash \{ 0\} $.} 
\end{align*}
Using the monotonicity formula (see e.g. Appendix A of \cite{Kuwert2004}), \eqref{eqn_willmore_bound} and \eqref{inversion2} we have that 
\begin{align*}
\mc {W}( \Sigma) + 4 \pi\theta^2 ( \mu , \infty) &= \mc {W}(\ov \Sigma) \leq \liminf _{ j\rightarrow \infty} \mc W ( \ov \Sigma _j) = \liminf _{ j \rightarrow \infty } \mc W (\Sigma_j).  
\end{align*} 
Additionally we also conclude that
\[
 \int_\Sigma |A|^2 d\mu \le C.
\]
Now as $ \ov \Sigma $ is smooth away from $0$ and $ \ov \Sigma _ j$ converges smoothly to $\ov \Sigma$ away from $0$ we conclude for $\rho$ sufficiently small and $j$ large (depending on $ \rho$) that $\partial B _ \rho (0)$ intersects $\ov \Sigma _j $ in a finite number of smooth closed curves, $C_{i,j},, i =1 ,\dots, k $ where $ \cup_{i = 1}^{k}C_{i,j}  = \ov \Sigma_ j \cap \partial B_\rho(0)$. We put 
\begin{align*}
\ov \Sigma _ \rho &= \ov \Sigma \backslash B_ \rho(0), \\
 \ov \Sigma _{ j , \rho , +} &= \ov \Sigma _j \backslash B _ \rho (0).
\end{align*}  
Considering appropriate triangulations we have that 
\begin{align*}
\chi ( \Sigma_j) = \chi ( \ov \Sigma _{ j ,\rho, +} ) + \chi ( \ov \Sigma_ j \cap B_ \rho(0)) \leq \chi ( \ov \Sigma _{ j, \rho, +}) + k, 
\end{align*}
as $ \chi ( \ov \Sigma _ j \cap B _ \rho(0) ) \leq k$. By smooth convergence away from zero we have that $ \ov \Sigma _{ j, \rho,+} \rightarrow \ov \Sigma_ \rho$ and in particular we have that $ \cup_{i=1}^{k}C_{i,j} =  \ov \Sigma_{j}\cap \partial B_\rho(0) \rightarrow \ov \Sigma\cap  \partial B_\rho(0)$ which shows that $\ov \Sigma\cap \partial B _\rho(0)$ consists of $k$ closed curves. As $ \int_{ \Sigma} |A|^{2} d \mu _{ \Sigma}< \infty$, we have that $ \Sigma $ is conformal to a compact Riemann surface $S$ with $l$ points $\{p_1,\ldots,p_l\}$ removed by a result of Huber \cite{Huber}. We let $\tilde{f}:S\backslash \{p_1,\ldots,p_l\} \to \Sigma$ be the conformal parametrisation. 

In particular we get that $ \ov \Sigma \cup \{x_0\}= I (\Sigma) \cup \{0\}$ is a closed Riemann surface. Note that as $\rho\rightarrow 0$ we see that $ B _ \rho (0) \rightarrow 0$ and hence the boundary curves $ f ^{-1} ( \ov {\Sigma} \cap \partial B _ \rho ( 0))\rightarrow \{ p _1 , \dots p_ l\}$ as $ \rho \rightarrow 0$, where $f:=I\circ \tilde{f}$. As the points have positive distance from each other we conclude that there are precisely $k$ points. Furthermore, by the local expansion \eqref{eqn_loc_expansion} (note that $\ov \Sigma$ is Willmore away from $0$), we see that locally about the origin, the surface consists of multivalued graphs. In particular, the $k$ closed curves in the set $f^{-1} ( \ov\Sigma \cap  \partial  B _ \rho (0 ))$ bound disks and hence this gives us 
\begin{align*}
\chi ( \ov \Sigma ) = \chi (\ov \Sigma _ \rho) + k = \lim_{ j \rightarrow \infty } \chi( \ov \Sigma _{ j,\rho,+} ) + k \geq \lim_{ j \rightarrow \infty} \chi( \Sigma_j), 
\end{align*}
and as a consequence, 
\begin{align*}
g( \ov \Sigma) \leq \liminf_{ j\rightarrow \infty } g ( \Sigma_j).
\end{align*}
\end{proof}

Now we are in a position to prove Theorem \ref{thm_main2}.  
 \begin{thm2}
Let $f_0: \mbb{S}^2 \rightarrow \R ^ 3 $ be a smooth immersion of a non-Willmore sphere with Willmore energy
\begin{align*} 
\mc{W} ( f_0) \leq 16 \pi.
\end{align*} 
If the maximal Willmore flow $f:\mbb{S}^2\times [0,T)\to \R^3$ with initial value $f_0$ does not converge to a round sphere then there exist sequences $ r _j, t _ j\nearrow T $ where
$ r_j \rightarrow \infty, r_j \rightarrow 0  $ or $ r_j \to 1 $ and a rescaled flow 
  \begin{align*}
 f _ j : \Sigma \times \left[ -\frac{ t_j}{ r_j^4} , \frac{ T - t_j} { r _j ^ 4 } \right) \rightarrow \R^3, f_j( p, \tau) = \frac{ 1}{ r_j} ( f ( p, t_j+ r_j^4 \tau) - x_j) 
\end{align*}
such that $f_j$ converges locally smoothly to either a catenoid, Enneper's minimal surface or a trinoid. In particular, if $\mc{W}(f_0) \leq 12\pi$ then either the maximal Willmore flow converges to a round sphere or $f_j$ converges locally smoothly to a catenoid.  
\end{thm2}
\begin{proof} 
As the initial immersion $f_0$ is not Willmore, we have that the Willmore energy must decrease immediately. Hence for $ t_0>0$, we have that $ \mc { W} (f (t_0)) < 16\pi $. Therefore we assume from now on that $  \mc { W} (f _0)\le 16\pi-\delta$ for some $\delta>0$. Moreover we assume that the maximal Willmore flow does not converge to a round sphere after a possible translation.
By the above discussion we then get a rescaled flow $f_j$ as claimed. More precisely, $f_j$ converges locally smoothly to a smooth, properly immersed Willmore surface $\hat{f}_0:\hat{\Sigma}_0\to \R^3$. 

Next we claim that $\hat{\Sigma}_0$ is complete and non-compact. This follows again from the above discussion except in the case $T=\infty$, $r_j\to 1$ in which the limit under translations could have a compact component. But then $\hat{\Sigma}_0$ would be diffeomorphic to $\mbb{S}^2$ and since $\mc{W}(\hat{f}_0)\le 16\pi-\delta$ this implies that $\hat{\Sigma}_0$ is a round sphere contradicting the fact that the flow does not converge to a round sphere.

This allows us to apply Lemma \ref{thm_convergence} and we conclude that there exists $x_0\notin \hat{\Sigma}_0$ such that $\overline{\Sigma} \cup \{x_0\}=I_{x_0}(\hat{\Sigma}_0) \cup \{x_0\}$ is a branched Willmore sphere with Willmore energy strictly less than $16\pi$. Combining Theorem \ref{thm_classify} and Remark \ref{energy} then shows that the divisor of $\tilde{f}=I_{x_0} \circ \hat{f}_0$ is equal to the conformal transform $\widehat{\tilde{f}}=mx_0$ with $m\in \{1,2,3\}$ and therefore $\hat{\Sigma}_0$ must be either a catenoid, Enneper`s minimal surface or a trinoid.
\end{proof}
Note that there exists smooth Willmore spheres with energy $ 16 \pi$. These were classified by Bryant \cite{Bryant1984} and correspond to the M\"obius transforms of complete minimal surfaces with four planar ends.   
\appendix
\numberwithin{equation}{section}
\section{The conformal Gauss map}
 
Here we review some results about the conformal Gauss map which are needed in section 3. All the results are based on the paper of  Bryant \cite{Bryant1984}, notes of Eschenburg \cite{Eschenburg} and a recent paper of Dall'Acqua \cite{Dallaqua}.

\subsection{Conformal Geometry of Spheres} 
For an immersion $f:\Sigma \to \R^3$ with normal $\nu:\Sigma \to \mbb{S}^2$ we define the central sphere $S_f(p)$ at $p\in \Sigma$ by
\begin{align*}
 S_f(p)= \begin{cases} \{ x\in \R^3:\,\ |x-m(p)|=r(p)\}\ \ & \text{if}\ \ H(p) \not= 0, \\
           \{ x\in \R^3:\,\ \langle x-f(p),\nu(p)\rangle =0\}\ \ & \text{if}\ \ H(p) = 0,
\end{cases}
\end{align*}
where $m(p)=f(p)+\frac1{H(p)} \nu(p)$ and $r(p)=\frac1{|H(p)|}$.

Next we consider the set $\mathbb{M}_0$ of unoriented spheres in $\R^3$ with centre $x_0$ and radius $r>0$ and we denote by $\phi:\R^4 \cup \{\infty \} \to \R^4 \cup \{\infty \}$ the inversion at the sphere $\partial B_{\sqrt{2}}((\underline 0,1))$, where $\underline 0=(0,0,0) \subset \R^3$. For each $S=\partial B_r(x_0))$ we define the map $F:\mathbb{M}_0 \to \R^4 \backslash \overline{B_1(0)}$, 
\[
 F(S)= \text{center of the sphere}\,\ \phi(\{p\in \R^4:\,\ |p-x_0|=r\}).
\]
A standard calculation shows that
\[
 F(\partial B_r(x_0)) =\frac1{|x_0|^2+1-r^2} (2x_0,|x_0|^2-1-r^2)
\]
and hence this map is not well-defined for $r^2=1+|x_0|^2$ since in this case the set $\phi(\{p\in \R^4:\,\ |p-x_0|=r\}$ is a plane.

Let us now consider $ \R ^ {5} = \R^ {4,1}$ with the Lorentzian product 
\begin{align*}
\langle X, Y\rangle = \sum_{ i=1} ^ 4 X_i Y_i - X_5 Y_5.
\end{align*} 
We also denote by $Q^4=\{q\in \R^5:\,\ \langle q,q \rangle=1\}$ the quadric in $\R^5$.

We calculate
\[
\langle ( F(\partial B_r(x_0)),1), (F(\partial B_r(x_0)),1) \rangle =\frac{4r^2}{(|x_0|^2-r^2+1)^2}.
\]
This shows that we can normalize and extend the map $F$ in order to get a new map $P:\mathbb{M}\to Q^4$,
\[
 P(S)=\frac1{r} \Big(x_0,\frac12(|x_0|^2-r^2-1),\frac12 (|x_0|^2-r^2+1)\Big),
\]
where $S=\partial B_r(x_0)$ and $\mathbb{M}$ is the set of oriented spheres with centre $x_0\subset \R^3$ and radius $r\in \R \backslash \{0\}$. The sphere is oriented by the inner (resp. outer) normal iff $r>0$ (resp. $r<0$).

Note that 
\[
 P_5(S)-P_4(S) =\frac1{r}\ \ \begin{cases} &>0 \ \ \ \text{for} \ \ r>0\\
                              &<0\ \ \ \text{for} \ \ r<0.
                             \end{cases}
\]
Hence spheres in $\R^3$ are represented by points in $Q^4$. We can also represent points $x\in \R^3$ by
\[
 X=\lim_{r\to 0} r P(\partial B_r(x))=\Big(x,\frac12 (|x|^2-1),\frac12(|x|^2+1)\Big).
\]
Note that the map $x\to X$ is an isometry onto 
\begin{align*}
L \cap \{ X_5 - X_4 = 1 \}.
\end{align*}
where $ L = \{ X \mid \langle X, X\rangle =0\} $ is the light cone in $ \R^5$.
 
An equivalent expression for $P$ can be gives as follows. We let $S\subset \mathbb{M}$ be a sphere (oriented by $\nu$) with mean curvature $H(x)$ with respect to $\nu(x)$ for $x\in S$. Then we have $S=\partial B_r(x_0)$ with $x_0=x+\frac1{H}\nu$, $r=\frac1{H}$ and we calculate
\[
P(S)=\Big(Hx+\nu,\frac{H}{2}(|x|^2-1)+\langle x,\nu\rangle, \frac{H}2(|x|^2+1)+\langle x,\nu \rangle \Big). 
\]
For $H\to 0$ we obtain a plane $S$ through the point $x$ with normal $\nu$ and we can extend the map $P$ to include this case by
\[
 P(S)=(\nu ,\langle x,\nu \rangle, \langle x,\nu \rangle). 
\]
With a slight abuse of notation we denote the induced map $\mathbb{M} \to \R P^4$ also by 
\[
 P(S)=[x_0,\frac12(|x_0|^2-r^2-1),\frac12 (|x_0|^2-r^2+1)].
\]
This map can again be extended to $\mathbb{M} \cup \R^3$ by using that for $S$ being a plane through $x$ with normal $\nu$ we have
\[
P(S):=[x,\frac12 (|x|^2-1),\frac12(|x|^2+1)].
\]

\subsection{Sphere Congruence} 
A sphere congruence is a smooth mapping $ S : \Sigma \rightarrow \mathbb{M}$ where $\Sigma$ is some $2$-dimensional manifold. Using the conformal geometry of spheres developed above, such a mapping can be represented by a smooth mapping $ Y : \Sigma\rightarrow Q^4, $ where
\begin{align*}
Y(m) = P(S_{r,m}).
\end{align*}  
where $ S_{r,m}$ is a sphere of radius $r$ which touches the point $m\in \Sigma$.  A smooth map $f:\Sigma\rightarrow \R^3$ is called an enveloping surface of the sphere congruence $S:\Sigma \to \mathbb{M}$ if 
\begin{enumerate}
\item $f(m) \in S(m) $ for all $m\in \Sigma$ and 
\item $d f _m ( T_m \Sigma) \subset T_{f(m)} S(m)  \quad \forall m \in \Sigma$. 
\end{enumerate}
Passing to the mapping $X: \Sigma\rightarrow L$, $X(m)=\Big(f(m),\frac12(|f(m)|^2-1),\frac12(|f(m)|^2+1)\Big)$, the conditions  $(1) $ and $(2)$ translate to 
\begin{align*}
\langle X(m), Y(m) \rangle =0 \quad \la d X(m), Y(m) \rangle = 0 \quad \forall m \in \Sigma, v \in T_m \Sigma. 
\end{align*}
So we may characterize the enveloping surface $X$ of a sphere congruence $Y$ by 
\begin{align*}
\langle X, Y \rangle =0, \quad \langle X, d Y \rangle = 0
\end{align*}
where we used $ d \langle X, Y \rangle = 0$. Hence $X$ is a null vector in $ T_{ Y(m) } Q ^ 4 $ which is normal to $Y$. Note by $\langle X, d Y\rangle = 0 $ we mean the vector in the same direction as $X$ at the point $Y$, i.e. $X_{Y} = Y + X, \langle X_Y, d Y \rangle = 0 $. This characterises a smooth enveloping surface, however, it requires no regularity of the sphere congruence $X$ itself. Furthermore this shows that for a space-like surface $\Sigma$ in $Q^4$ the normal bundle has signature $ ( +,-)$ and that each null vector field corresponds to a sphere congruence.    

\subsection{Conformal Gauss map} 

For an immersion $ f : \Sigma \rightarrow \R^3$ with unit normal $\nu: \Sigma\rightarrow \mbb{S}^2$ and mean curvature $H$ we define the conformal Gauss map by $Y:\Sigma \to Q^4$, $Y(m)=P(S_f(m))$. We have
\[
 Y=H X+N,
\]
where $X=\Big(f,\frac12(|f|^2-1),\frac12 (|f|^2+1)\Big)$ and $N=(\nu,\langle f,\nu \rangle,\langle f,\nu \rangle)$.

The map $Y$ satisfies
\begin{align*}
\langle \nabla _ i Y , \nabla_j Y\rangle = ( H^2 - K) \langle \nabla _i f , \nabla _j f \rangle 
\end{align*} 
and hence it is a conformal map (away from umbilic points) with respect to the conformal structure on $\Sigma$ induced by the immersion $f:\Sigma \to \R^3$ and the Willmore functional is the area of $Y$
\begin{align*}
\mc {W} (f) = \Area( Y).
\end{align*}
In particular, $f$ is a Willmore immersion if $Y$ is a minimal surface (in $Q$). The converse is true, however there is some subtlety. Namely $Y$ is an immersion into a Lorentzian manifold so some deformations maybe timelike, but deformations that come from the deforming Willmore surface must be spacelike. 

Let $\Sigma$ be a surface, possibly non-compact and let $ f: \Sigma\rightarrow \R^3$ be an immersion. Then $f$ is a Willmore immersion if and only if $ \mc { W } ( f, \Sigma') $ is stationary for any relatively compact open subset $ \Sigma' \subset \Sigma$. i.e. any smooth variation $ f ^t $ of $f$ with $f^t = f $ outside of $\Sigma'$. In this case we have that 
\begin{align*}
\delta \mc{W}(f, \Sigma')= \frac{ d}{dt} \bigg|_{ t=0}  \mc { W}(f^t, \Sigma') = 0.  
\end{align*}   

The following are equivalent
\begin{enumerate}
\item $f$ is a Willmore immersion 
\item $Y$ is a conformal harmonic map .
\end{enumerate}

Next we recall the standard equations for an immersion using complex co-ordinates $z$ for the conformal structure defined by the induced metric $g$ of the immersion $f:\Sigma \to \R^3$ with unit normal $\nu:\Sigma \to \mbb{S}^2$. We have the local expression $g=e^\lambda |dz|^2$. Moreover the second fundamental form $A=-\langle d\nu, df \rangle$ can be written as
\[
 A =   \Re \{ \varphi d z ^ 2 + H e ^  \lambda dz d \ov z  \},
\]
where $\varphi=2 (f_{zz},\nu)$ is the Hopf differential.

Now let $ \Lambda $ be any symmetric $m$ -form on $\Sigma$. 
\begin{align*}
\Lambda = \sum_{ j+k = m } \Lambda_{ jk } d z ^ j d \ov z ^ k.
\end{align*}
This decomposition is invariant under change of holomorphic charts and we have that 
\begin{align*}
\Lambda ^{ j,k} = \Lambda _{jk} d z ^ j d \ov z ^ k
\end{align*}
is the $ (j,k)$ part of the $\Lambda$.

Associated with the conformal Gauss map is a holomorphic quartic differential.
\begin{proposition}
Let $ Y : \Sigma \rightarrow Q ^{4} \subset \R^5 $ be a conformal minimal immersion, then 
\begin{align*}
\la \alpha , \alpha \ra ^{ (4,0)} = \la Y_{zz}, Y_{zz} \ra d z ^4, 
\end{align*} 
where $\alpha:T\Sigma \otimes T\Sigma \to NY$ is the second fundamental form with components $\alpha_{ij}=(Y_{ij})^\perp$.

Moreover we have that $ \la \alpha, \alpha \ra ^{ (4,0)} $ is holomorphic.
\end{proposition}

Using the local complex coordinates from above one can derive an expression for the quartic form.
\begin{lemma}
In local complex co-ordinates we have the following expression for the quartic form $\langle \alpha, \alpha \rangle ^{ (4 , 0) },$ 
\begin{align}\label{eqn_main}
\la Y_{ zz} , Y_{ zz} \ra =  \varphi H_{ zz} - H_ z ( \varphi e ^{ - \lambda }  )_z e ^ \lambda + \frac{ \varphi ^ 2 H^2 } { 4}.
\end{align}
\end{lemma}

Since $NY$ is a 2 dimensional bundle we can choose a frame $ \{N_1, N_2 \} $ so that 
\begin{align*}
\la N_1, N_2 \ra =1, \la N_1, N_1 \ra = \la N_2, N_2 \ra =0.
\end{align*}
A standard calculation shows that 
\begin{align*}
 \la Y_{ zz}, N_1 \ra _{ \ov z } &= \la N_{ 1 \ov z }, N_2 \ra \la Y_{zz}, N_1 \ra ,\\
  \la Y_{ zz}, N_2 \ra _{ \ov z } &= \la N_{ 2 \ov z } , N_1 \ra \la Y_{ zz}, N_2 \ra  = - \la N_{ 1 \ov z } , N_2 \ra \la Y_{zz}, N_2 \ra \ \ \text{and}\\
\la Y_{zz} ,Y_{zz} \ra &= \la Y_{zz},N_1 \ra \la Y_{zz}, N_2 \ra. 
\end{align*}
As the first two equations are Cauchy-Riemann type equations, this implies that the functions $ \la Y_{ zz}, N_1 \ra $ have isolated zeroes or vanish everywhere.

A similar result is true for $ \la Y_{ zz}, Y_{zz} \ra $. More precisely we have that if $ \la Y _{zz} , Y _{ zz} \ra = 0$ then  
\begin{align}
N_{ j z } = f N_{ j } \label{prop_hol}
\end{align} 
for some function $f$ and either $ j=1$ or $j=2$.

This shows that the mapping $ N_ i \mapsto [ N_i ] \in \mbb{CP}^4 $ is anti-holomorphic. Note that in the case where $NY$ has a metric of signature $ ( +,-)$ we can choose $ N_1, N_2 $ to be real vectors. Hence if $ [ N_i]$ is holomorphic, then from the Cauchy-Riemann equations we get that $ [N_i ]$ is a constant.

\section{Weierstrass-Enneper representation of complete, branched minimal surfaces}

The results of this appendix are straightforward consequences of the corresponding statements in \cite{Bryant1984}.

By the Weierstrass-Enneper representation, it is well known that the classification of the branched complete minimal surfaces of finite total curvature may be reduced to an algebraic problem. In particular, we know that (as before $D$ is the divisor)
\begin{align*}
\partial ( I _ {\hat f} \circ f ) = \left ( 
\begin{array}{c} 
\phi_1 \\
\phi_2 \\
\phi_3
\end{array} 
\right)
\end{align*}  
where $ \phi_i$ are meromorphic one forms on $\Sigma$. Our geometric data translates into homomorphic data as follows
 \begin{enumerate} 
 \item $I _ { \hat f} \circ f  $ is an immersion $\iff$ the $ \phi_ i$ have no common zeroes. 
 
 \item $ I _ { \hat f} \circ f $ has a branch point of order $k$ $ \iff$ the $\phi_i$ have a common zero of order $k$ $\iff$ $ \phi$ are holomorphic sections of the bundle $K _{\Sigma} \otimes [ - D']$ where $ K_{\Sigma}$ is the canonical bundle of $ \Sigma$ with the given complex structure and $D'=\sum_{p\in D}(m(p)-1)p$.   
\item $ I _ { \hat f} \circ f $ is conformal $ \iff$  $\phi_1 ^ 2 + \phi_2 ^ 2 + \phi_ 3^ 2 = 0 $. 
\item The ends of $ I _ { \hat f} \circ f $ are embedded $\iff$ the $ \phi_i$ have poles of at worst second order on $D$ $\iff$ $\phi_i $ are holomorphic sections of $K_\Sigma \otimes [2D]$ where $ K_\Sigma$ is the canonical bundle of $ \Sigma$ with the given complex structure  
\item Then ends of  $ I _ { \hat f} \circ f $ have branch order $k(a)$ at $a \in D$ $\iff$ the $\phi_i$ have poles of at worst order $k(a)+1$ at $a \in D$ $\iff$ the $ \phi_i$ are holomorphic sections of $ K_{M} \otimes[ D_1]$ where $ D_ 1 = \sum_{a\in D} (k(a)+1) a$.    
\item The ends of $ I _ { \hat f} \circ f $ are planar $ \iff$ the $ \phi$ are differentials of the second kind that is $ \Res_{m } \phi_i = 0 $ for all ends $a_ i $ and $ m \in D$
\item $ I _ { \hat f} \circ f $ is single valued on $\Sigma^*$ $\iff$ for all $\gamma   \in H^0 ( \Sigma, \mbb Z ), \Re ( \Per_{\gamma } \phi_i ) = 0$. 
 \end{enumerate}

\end{document}